\documentclass[11pt]{amsart}

\usepackage{amsmath}
\usepackage{amsfonts}
\usepackage{amssymb}
\usepackage{eucal}
\usepackage[T1]{fontenc}

\def\dfrac{\displaystyle \frac}

\newtheorem{theorem}{Theorem}
\newcounter{lemm}
\newtheorem{lemma}[lemm]{Lemma}
\newcounter{prop}
\newtheorem{proposition}[prop]{Proposition}
\newcounter{coro}
\newtheorem{corollary}[coro]{Corollary}

\newcounter{rema}

\newcounter{exam}
\newenvironment{example}{\smallskip\noindent\stepcounter{exam}{\bf Example \arabic{exam}}.}{\qed\smallskip}

\begin{document}
\date{\today}
\vspace*{30mm}
\title{Sewing cells in almost cosymplectic and almost Kenmotsu geometry}
\author{Piotr Dacko}
%\address{Institute of Mathematics and Computer Science, Wroc{\l}aw University of Technology, Wybrze\.{z}e Wyspia\'nskiego 27, 50-370 Wroc{\l}aw, Poland}
\email{{\tt piotrdacko{\char64}yahoo.com}}
\begin{abstract}For a finite family of 3-dimensional almost contact metric manifolds with closed 
the structure form $\eta$ is described a construction of an almost contact metric manifold, where the 
members of the family are building blocks - cells. Obtained manifold share many properties of cells. One 
of the more important are nullity conditions. If cells satisfy nullity conditions - then - in the 
case of almost cosymplectic or almost $\alpha$-Kenmotsu manifolds - ``sewed cells'' also satisfies 
nullity condition - but generally with different constants. It is important that even in the case of the 
generalized nullity conditions - ``sewed cells'' are the manifolds which satisfy such conditions provided the 
cells satisfy the generalized nullity conditions. 
\end{abstract}
\maketitle
% \include{./intro}
% \nopagebreak
% \include{./prelim} 
% \include{./fmanif}
% \include{./extrgeom}
% \include{./almcosken}
% \include{./exfinrem}
% \include{./biblio}
% \end{document}

\section{Introduction}
The recent years  witnessed the very extensive study of the geometry of  almost 
contact metric manifolds. One of the most important results are several classification theorems 
concerning contact metric, almost cosymplectic and almost $\alpha$-Kenmotsu manifolds which satisfy 
the nullity conditions. Non-normal contact metric manifolds are classified finally by E. Boeckx
{\cite{BlKouPap,Boeckx}} up to the equivalence relation defined by the $D$-homoteties. 
The class of almost cosymplectic manifolds was studied by the author and Z. Olszak {\cite{DacXi,DacOlszMu,DacOlszKapMuNu}} -
now almost cosymplectic manifold are classified up to $D$-conformal deformations of the structure. Just 
recently M. A. Pastore, D. Dileo and V. Saltarelli  {\cite{DilPastNulDist,Dileo1}} resolved the problem 
of the classification of an almost Kenmotsu manifolds which satisfy nullity or generalized nullity condition - 
with only one small gap remaining - the local description of generalized $(\kappa,\mu)$-nullity almost Kenmotsu 
manifolds. Successfully just recently V. Saltarell classified 3-dimensional such manifolds. 
 One of the problems in that direction - to consider higher dimensions-  not many examples were known. 
One of the main goal of this paper is to  resolve this problem.

Let $M_1$,\ldots, $M_k$ are almost contact metric 3-dimensional manifolds. On the Cartesian product $M=M_1\times\ldots\times M_k$
 we introduce - quite naturally - an almost metric $f$-structure. We are interested in $2k+1$ submanifolds 
of the manifold $M$ we call - the manifolds of sewed cells. Here the manifolds $M_1$,\ldots, $M_k$ are ``cells'' which are 
``sewed'' together to create just mentioned submanifolds of ``sewed cells'' cf. Section 4.

It appears that the manifold $N$ of sewed cells is enough ``neatly'' embedded into the product $M_1\times\ldots\times M_k$ to share 
many properties of cells. Here of the particular interest are nullity conditions. In the Section 5 we prove that sewed cells 
of almost cosymplectic, almost $\alpha$-Kenmotsu 3-dimensional manifolds is again the manifold of the same class. 
Moreover the nullity conditions are inherited by the sewed cells - even in the case of the generalized nullity conditions. 

However the class of manifolds obtained by the presented method is much wider than almost cosymplectic and almost $\alpha$-Kenmotsu.
Therefore the concept of ``sewed cells''  provides a wide range of new and interesting explicit examples of almost contact metric manifolds.

\section{Preliminaries}

An almost contact metric structure on a manifold $M$ is a quadruple $(\varphi,\xi,\eta,g)$ of the tensor fields, where 
$\varphi$ is an affinor (a $(1,1)$-tensor field), $\xi$ a vector field, $\eta$ a one-form and $g$ 
a Riemannian metric, such that
\begin{eqnarray*}
 \varphi^2 &=& -Id +\eta\otimes\xi, \quad \eta(\xi)=1,  \\
  g(\varphi X,\varphi Y) &=& g(X,Y)-\eta(X)\eta(Y), 
\end{eqnarray*}
here the entries $X$, $Y$ are vector fields. The manifold $M$ endowed with an almost contact metric structure is 
called an almost contact metric manifold. The definition follows that a tensor field $\varPhi(X,Y)=g( X,\varphi Y)$ 
is a totally skew-symmetric, i.e. a 2-form on $M$ called a fundamental form {\cite{Blair}}.

Let $M\times \mathbb{R}$ be the Cartesian product of an almost contact metric manifold 
and the real line. Let define an  almost complex structure $J$ on $M\times\mathbb{R}$ 
\begin{equation*}
 J(X,f\frac{d}{dt}) = (\varphi X-f\xi,\eta(X)\frac{d}{dt}),
\end{equation*}
 with respect to  the canonical splitting $T(M\times\mathbb{R})=TM\oplus T\mathbb{R}$. If $J$ is integrable, i.e. 
 $J$ is a complex structure - $M\times \mathbb{R}$ is a complex manifold, then 
the almost contact metric manifold $M$ is called normal - and the structure is called normal. 

The Nijenhuis torsion tensor field $[\varphi,\varphi]$ of the  structure $\varphi$ is defined by
\begin{equation*}
 [\varphi,\varphi](X,Y) = \varphi^2[ X, Y]+[\varphi X,\varphi Y]-\varphi[\varphi X,Y]-\varphi[X,\varphi Y].
\end{equation*}
The manifold $M$ is normal if and only if {\cite{Blair}}
\begin{equation*}
 [\varphi,\varphi]+2d\eta\otimes\xi=0.
\end{equation*}

An almost contact metric manifold $M$ which satisfies $d\eta=0$ and $d\varPhi=0$, both $\eta$ and the 
fundamental form are closed is called an almost cosymplectic manifold {\cite{GoldYano}}. The manifold $M$ 
is called an almost $\alpha$-Kenmotsu manifold, $\alpha$ a real $\neq 0$ constant, if $d\eta=0$ and $d\varPhi=2\alpha\eta\wedge\varPhi$. In the paper {\cite{Murath}} almost cosymplectic
and almost $\alpha$-Kenmotsu manifolds are studied from the common point of view - and they are called 
almost $\alpha$-cosymplectic manifolds, $\alpha$ arbitrary real constant. A normal almost cosymplectic manifold 
is called cosymplectic, similarly we have $\alpha$-Kenmotsu manifolds. The local structure of cosymplectic,
$\alpha$-Kenmotsu manifolds, now is very well understood. 
An almost cosymplectic manifold is cosymplectic iff $\nabla\varphi=0$, i.e. $\varphi$ is a covariant 
constant with respect to the Levi-Civita connection. From the other hand Goldberg and Yano  {\cite{GoldYano} proved
that the conditions 
\begin{equation*}
 \nabla\varphi=0, \quad R(X,Y)\varphi Z=\varphi R(X,Y)Z,
\end{equation*}
are equivalent on an almost cosymplectic manifold. Therefore an almost cosymplectic manifold is cosymplectic iff the structure $\varphi$ commutes 
with the Riemann curvature.

One of the most important geometric quantities on an almost contact metric manifold are affinors 
\begin{equation*}
 h=\frac{1}{2}\mathcal L_\xi\varphi, \quad h'=h\varphi,
\end{equation*}
 $h$ measure the rate of the change of the tensor $\varphi$ under the flow generated by the vector field $\xi$,
for normal manifolds $h=0$ identically. 

Leaving the explanations of the genesis of the concept, we say that 
an almost contact metric manifold $M$ satisfies a $(\kappa,\mu,\mu')$-nullity condition, or equivalently
that the vector field $\xi$ belongs to a $(\kappa,\mu,\mu')$-nullity distribution $(\kappa,\mu,\mu')\in \mathbb{R}^3$ 
if
\begin{equation*}
\begin{array}[]{rcl}
 R(X,Y)\xi &=& \kappa (\eta(Y)X-\eta(X)Y)+\mu(\eta(Y)hX-\eta(X)hY)+ \\[+4pt]
            &&  +\mu'(\eta(Y)h'X-\eta(X)h'Y).
\end{array}
\end{equation*}
Although in the definition we require $\kappa$, $\mu$, $\mu'$ constant in the case of almost cosymplectic manifolds and 
almost Kenmotsu manifolds it appeared convenient  to extend the definition in the direction where $\kappa$, $\mu$, $\mu'$ 
are some functions. Such weaker conditions are called generalized nullity conditions. The papers 
{\cite{Boeckx,DacXi,DacOlszKapMuNu,DacOlszMu,DilPastNulDist,Dileo1,PastSalt,Salt} concern with 
the classification theorems for particular classes of almost contact metric manifolds.

In this paper by a cell is understood a 3-dimensional almost contact metric manifold
$M^3$, $(\varphi,\xi,\eta,g)$ such that $d\eta=0$. Cells are denoted by $C_1$, $C_2$, etc.

We finish the section with the following 
\begin{proposition}
\label{cellfeat}
 On a cell $C=(M^3,\varphi,\xi,\eta,g)$ the vector field $\xi$ is geodesic $\nabla_\xi\xi=0$ and the structure 
$\varphi$ is $\xi$-parallel
\begin{equation*} 
 \nabla_\xi\varphi =0.
\end{equation*}
\end{proposition}
\begin{proof}
The formula for the covariant derivative $\nabla\varphi$ from  the Lemma 6.1 in \cite{Blair} in the case $d\eta=0$ will take a shape
\begin{equation*}
 g((\nabla_X\varphi)Y,Z)=3d\varPhi(X,\varphi Y,\varphi Z)-3d\varPhi(X,Y,Z) +g([\varphi,\varphi](Y,Z),\varphi X),
\end{equation*}
note $d\varPhi=\alpha\eta\wedge\varPhi$ for a function $\alpha$ on $M^3$, therefore 
\begin{equation*}
g((\nabla_\xi\varphi)Y,Z) = 3d\varPhi(\xi,\varphi Y,\varphi Z)-3d\varPhi(\xi,Y,Z)  = \alpha(\varPhi(\varphi Y,\varphi Z)-\varPhi(Y,Z)) =0,  
\end{equation*}
and $\nabla_\xi\varphi=0$ which implies $\nabla_\xi\xi=0$.
\end{proof}

\section{An almost metric $f$-structure of a cell product}

An affinor $f$ on a manifold $M$, such that $f^3+f=0$  is called an almost $f$-structure. The existence of an almost $f$-structure 
determines a reduction of a structure group of the manifold. If this $G$-structure is integrable, the affinor $f$ is called 
a $f$-structure. Examples of almost $f$-structures are almost complex structures and almost contact structures. If there is 
a Riemannian metric $\bar g$ on the manifold $M$ such that $\bar g(fX,Y)=-\bar g(X,fY)$ then a triple $(M,f,\bar g)$ is called 
an almost metric $f$-manifold {\cite{FalPastFram,YanoIshi,YanoKon}}.

It is not stated explicitly in the above definition but it is assumed that distributions 
\begin{equation*}
 \begin{array}{c}
 Ker(f):p\mapsto \lbrace X\in T_pM| fX=0\rbrace, \\[+4pt]
Im(f):p\mapsto \lbrace X\in T_pM| X=fY \rbrace
\end{array}
\end{equation*}
the kernel of $f$ and the image have constant dimensions. The distribution $Im(f)$ is always even-dimensional 
$\text{dim}\, Im(f) =2l$ 
and the restriction $J=f|_{Im(f)}$ defines a formal almost complex structure, i.e. $J^2X=-X$ whenever $X\in \Gamma(Im(f))$ is a section of $Im(f)$. The distributions $Ker(f)$, $Im(f)$ are complementary and orthogonal with res. to the metric $\bar g$. 

Let $(M,f,\bar g)$ be an almost metric $f$-manifold. A framing $(\bar\xi_1,\ldots,\bar\xi_k)$ is a repair 
of orthonormal vector fields which spann $Ker(f)$, here $k=\text{dim}\,Ker(f)$. A coframing is a corepair of 
dual one-forms $(\bar\eta_1,\ldots,\bar\eta_k)$, $\bar\eta_i(\bar\xi_j)=\delta_i^j$, $\delta_i^j$ - the Kronecker's 
$\delta$. Note that arbitrary $\bar\xi_i$ is perpendicular to the image $Im(f)$, moreover $Im(f)$ coincides with 
a common kernel of the forms $\bar\eta_i$, i.e. $Im(f) = ker\,\bar\eta_1 \cap \ldots \cap ker\,\bar\eta_k$.

An almost metric $f$-manifold is called a $C$-manifold if near each point $p\in M$ there is a closed coframing, i.e. 
each form $\bar\eta_i$ is closed $d\bar\eta_i=0$. In a consequence the distribution $Im(f)$ is involutive, hence 
completely integrable and a leaf $N\subset M$ is an almost complex submanifold.  

Let $C_i=(M_i^3,\varphi_i,\xi_i,\eta_i,g_i)$, $i=1,\ldots,k$  be a finite family of cells, 
$M=C_1\times C_2\times\ldots \times C_k$ be the Cartesian product. By $\pi_i:M\rightarrow C_i$ we denote 
the canonical projections, thus if $q=(p_1,\ldots,p_k)\in M$ then $\pi_i(q)=p_i\in C_i$. The distributions
$D_1,D_2,\ldots,D_k$ are defined by 
$$  D_i= \bigcap\limits_{j\neq i}Ker\,\pi_{*j}, $$
where $\pi_{*j}:TM\rightarrow TC_j$ denotes the tangent map. We have the canonical splitting $TM=D_1\oplus\ldots\oplus D_k$.
Let $X$ be a vector field on a cell $C_i$. By a lift of the vector field $X$ we mean a vector field $\bar X$ on the 
manifold $M$, such that $\pi_{*i}(\bar X)=X$ and $\bar X\in \Gamma(D_i)$.  

\begin{proposition}
 For a vector field $X$ on a cell $C_i$, there exists a lift $\bar X$ and is determined 
 uniquely.
\end{proposition}
 
\begin{proof}
 Assume $\bar X_1$, $\bar X_2$ are lifts of the vector field $X$, at every point where the lifts are defined
$\pi_{*i}(\bar X_1-\bar X_2) = 0$, therefore $\bar X_1-\bar X_2$ belongs to the kernels of all 
projections $\pi_{*\cdot}$'s, so must vanishes identically. This proves the uniqueness. In consequence
it is enough to prove the existence only locally.
Now let fix a point $q=(p_1,\ldots,p_k)$, $x=\pi_i(q)=p_i$
and let a neighborhood $U_x$ of the point $x$ be such that a local flow $exp(tX)$ generated by $X$ 
exists on $U_x$. Near a point $q\in M$ we define a local flow $s_t$ on $\bar U_q=\pi_i^{-1}(U_x)$
\begin{equation}
 s_t(p_1,\ldots,p_i,\ldots,p_k) = (p_1,\ldots, exp(tX)p_i, \ldots, p_k).
\end{equation}
From the definition $\pi_i\circ s_t=exp(tX)$. Let $\bar X$ be an infinitesimal generator of $s_t$, for $p\in \bar U_q$
$$
X_{\pi(p)} = \frac{d}{dt}\pi_i(s_t p)|_{t=0} = \pi_{*i}(\frac{d}{dt}s_t p|_{t=0}) = \pi_{*i}(\bar X_p).
$$
Finally for $j\neq i$, $\pi_j(s_t p)= const$, $\pi_{j*}(\bar X)=0$, the vector field $\bar X$ belongs to the kernels of $\pi_j$, 
therefore it is a section of the distribution $D_i$. 
\end{proof}

If $X_1$, $X_2$ are vector fields on $C_i$ and $\tau_1$, $\tau_2$ are functions then the lift of the combination
$\tau_1 X_1+\tau_2 X_2$  is a vector field $\tau_1^*\bar X_1+\tau_2^*\bar X_2$, where $\tau_{1,2}^*=\tau_{1,2}\circ\pi_i$.

For the cell $(C_i,\varphi_i,\xi_i,\eta_i,g_i)$ we define a tensor field $\bar\varphi_i$ (a lift) on the product $M=C_1\times\ldots \times C_k$ as follows: let 
\begin{equation*}
 \varphi_i = \sum\limits_{k,j=1}^3\varphi_{ki}^j\alpha^k\otimes X_j,
\end{equation*}
be a local description of the tensor $\varphi_i$ with res. to a local repair $(X_1,X_2,X_3)$ on $C_i$, here 
$\varphi_{ki}^j$ are smooth functions on $C_i$. Then by the definition 
\begin{equation*}
 \bar\varphi_i = \sum\limits_{k,j=1}^3\bar\varphi_{ki}^{j}\bar\alpha^k\otimes\bar X_j,
\end{equation*}
where $\bar\varphi_{ki}^{j}=\varphi_{ki}^j\circ \pi_j$ are functions on $M$, the forms $\bar\alpha^k=\pi_i^*\alpha^k$ are 
pullbacks and $\bar X_j$ are lifts of the vector fields $X_j$. The tensor field $\bar \varphi_i$ can be characterized as follows:
if $\bar X$ is a lift, then $\bar\varphi_i\bar X$ is a lift of the vector field $\varphi_i X$, $\bar\varphi_i\bar X=\overline{\varphi_i X}$.

\begin{theorem}
 A pair $(f,g)$ of the tensor fields
\begin{equation}
 f = \sum\limits_{i=1}^k\bar\varphi_i,
\end{equation}
and $\bar g$ - the Riemannian product metric
\begin{equation}
 \bar g = \sum\limits_{i=1}^k \pi_i^*g_i.
\end{equation} 
defines an almost metric $f$-structure on the cell product, which is globally framed. Moreover $(M,f,g)$ is a C-manifold. 
The fundamental form $\bar\varPhi(X,Y)=\bar g(X,fY)$ is given by 
\begin{equation}
 \bar\varPhi =\sum\limits_{i=1}^k\pi_i^*\varPhi_i, 
\end{equation}
 i.e. is the sum of the pullbacks of the fundamental forms of the cells.
\end{theorem}
\begin{proof}
 From the definition of $\bar\varphi_i$ it follows that $\bar\varphi_i\bar\varphi_j=\bar\varphi_j\bar\varphi_i=0$ for 
$i\neq j$, and $\bar\varphi_i^3+\bar\varphi_i=0$, $i=1,\ldots,k$. Therefore
\begin{equation}
 f^3+f= \sum\limits_{i=1}^k (\bar\varphi_i^3+\bar\varphi_i)=0.
\end{equation}
  With res. to the product metric the decomposition
$TM=D_1\oplus\ldots\oplus D_k$ is an orthogonal decomposition, i.e. $D_i$, $D_j$ are pairwise orthogonal,
each projection $\pi_i$ is a Riemannian submersion, moreover if 
$\bar X\in \Gamma(D_i)$, $\bar Y\in \Gamma(D_j)$  are  lifts  then 
\begin{equation}
\begin{array}{l}
\bar g(\bar X, \bar Y) =  
\begin{cases}
 g_i(X,Y)\circ\pi_i, & i=j \\
0 , & i \neq j,
\end{cases} \\ [+8pt]
\bar g(f \bar X,\bar Y) = \bar g(\bar\varphi_i \bar X,\bar Y) = 
 \begin{cases}
 g_i(\varphi_i X,Y)\circ\pi_i,& i=j \\ 0, & i\neq j,
\end{cases}
\end{array}
\end{equation}
hence $f$ is skew-symmetric.
For the lifts $\bar\xi_i$
\begin{equation*}
 f\bar\xi_i=\bar\varphi_i\bar\xi_i=\pi_{i*}(\varphi_i\xi_i)=0.
\end{equation*}
Let a vector field $U$ be orthogonal to a distribution spanned by $(\bar\xi_1,\ldots,\bar\xi_k)$. Assume that 
$fU=0$ for arbitrary point but at a point $q$, $U_q\neq 0$.
Let $U^i$ be an orthogonal projection on the distribution $D_i$, such that $U^i_q\neq 0$.  
We extend a vector $V_x=\pi_{i*}(U^i_q)$, $x=\pi_i(q)$  to a local vector field 
$V$, $V\perp\xi_i$. For the lift $\bar V$, $f\bar V_q=\bar\varphi_i\bar V$, 
and by assumption $0=\pi_{*i}(\bar\varphi_i\bar V)=\varphi V$ at $x=\pi_i(q)$, implies $V=0$ and $\bar V_q=U^i_q=0$,
the contradiction. Therefore
\begin{equation*}
 Ker(f) = Spann(\bar\xi_1,\ldots,\bar\xi_k).
\end{equation*}
Pullbacks $(\bar\eta_1,\ldots,\bar\eta_k)$ define a dual closed coframing for $\bar\eta_i=\pi_i^*\eta_i$ and all forms $\eta_i$ are closed. 
 \end{proof}

\section{An extrinsic Riemannian geometry of a manifold of sewed cells}
In the present paper we are merely interested  in the structure of a product of cells 
but rather in a very particular submanifolds.

Let $(M,f,\bar g, \bar\xi_1,\ldots,\bar\xi_k)$ be a product of cells $k \geqslant 2$,
$M=C_1\times\ldots\times C_k$, $C_i=(C_i,\varphi_i,\xi_i,\eta_i,g_i)$  with 
its canonical almost metric $f$-structure, and canonical global framing $(\bar\xi_1,\ldots,\bar\xi_k)$, 
defined by the lifts of the vector fields $\xi_i$. Let define a median vector field $\bar\xi$ on $M$
\begin{equation*}
 \bar\xi = \dfrac{\bar\xi_1+\ldots+\bar\xi_k}{\sqrt{k}},
\end{equation*}
the median $\bar\xi$ is globally defined and $\bar g(\bar\xi,\bar\xi)=1$.
\begin{proposition}
\label{invfxi}
A distribution $Im(f)\oplus \mathbb{R}\bar\xi$ is involutive. 
\end{proposition}
\begin{proof}
Let $X\in \Gamma(Im(f))$
$$
0=2d\bar\eta_i(\bar\xi_j,X)=-\bar\eta_i([\bar\xi_j,X]),\quad i,j=1,\ldots,k,
$$ 
therefore $[\bar\xi_j,X]\in \Gamma(Im(f))$ 
\begin{equation*}
[\bar\xi,X] = \frac{1}{\sqrt{k}}\sum\limits_{i=1}^k[\bar\xi_i,X] \in \Gamma(Im(f)),  
\end{equation*}
if $X$ is an arbitrary vector field tangent to $Im(f)\oplus \mathbb{R}\bar\xi$, then $X=\bar X+ \tau\bar\xi$, 
$\bar X\in \Gamma(Im(f))$, $\tau$  is a function on $M$, and 
$
[\bar\xi,X] = [\bar\xi,\bar X]+(\bar\xi\tau)\bar\xi \in \Gamma(Im(f)\oplus\mathbb{R}\bar\xi
$.
\end{proof}

By the  Proposition \ref{invfxi} through any point of $M$ is passing a unique integral submanifold $N\subset M$ 
of $Im(f)\oplus\mathbb{R}\bar\xi$ we shall call a manifold of sewed cells $N=C_1-C_2-\ldots-C_k$.  

Let $q_0\in M$, $\pi_i(q_0)=x^i_0\in C_i$, let $x^i_0\in U^i_0\subset C_i$ be an open disc such small there is 
a function $\tau_i:U^i_0\rightarrow \mathbb{R}$, $d\tau_i=\eta_i|_{U^i_0}$, $\tau_i(U^i_0)=(-\epsilon_i,\epsilon_i)$,
$\tau_i(x^i_0)=0\in \mathbb{R}$.
A set $q_0\in U_0=\bigcap_{i=1}^k\pi_i^{-1}(U^i_0)$ 
with functions $\bar\tau_1=\tau_1\circ\pi_1$,..,$\bar\tau_k=\tau_k\circ\pi_k$ we will call a polidisc centered at $q_0$, 
from the definition $\bar\tau_1(q_0)=\ldots=\bar\tau_k(q_0)=0$.
Let $\Delta_0 \subset U_0$  be a connected component of a set
\begin{equation*}
 \Delta =\lbrace q\in U_0\; |\; \bar\tau_1(q)=\bar\tau_2(q)=\ldots=\bar\tau_k(q) \rbrace, 
\end{equation*}
containing $q_0$, (all the functions $\bar\tau_i$ attain the same value on a point of $\Delta_0$).
 If we define a local map $U_0\rightarrow\mathbb{R}^k$, 
$U_0\ni q\mapsto (\bar\tau_1(q),\ldots,\bar\tau_k(q))\in \mathbb{R}^k$ then the image $\hat\Delta$ lies on the diagonal 
$\lbrace (t,\ldots,t)\subset \mathbb{R}^k\rbrace$.

For  a section  $X=\bar X+\tau\bar\xi$ of $Im(f)\oplus\mathbb{R}\bar\xi$, 
\begin{equation}
\label{alletas}
\bar\eta_1(X)=\bar\eta_2(X)=\ldots=\bar\eta_k(X)=\frac{\tau}{\sqrt{k}}.
\end{equation}
Let fix a point $q_0\in M$, $N=C_1-\ldots-C_k$ are sewed cells through $q_0$, and $(U_0, \bar\tau_1,\ldots,\bar\tau_k)$ is a polidisc at $q_0$,
and $N_0\subset N\cap U_0$ be an embedded, connected, simply connected part of $N$ containing $q_0$. Let 
$\iota:N_0\subset U_0$ denote an inclusion map.  For a point $p\in N_0$,
$\gamma:[0,1]\rightarrow N_0$ is a smooth curve joining the points $p$ and $q_0$, $\gamma(0)=q_0$, $\gamma(1)=p$.
Then $r=(\tau_1\circ\gamma,\ldots,\tau_k\circ\gamma)$ is a curve in $\mathbb{R}^k$
\begin{equation*}
 \dot r = (\bar\tau_{1*}\circ\dot \gamma,\ldots,\bar\tau_{2*}\circ\dot \gamma)
= (\bar\eta_1(\iota_*\dot\gamma),\ldots,\bar\eta_k(\iota_*\dot\gamma)),
\end{equation*}
by (\ref{alletas}) components of the tangent vector $\dot r$ are equal to each other, therefore the curve $r(s)$ itself  must lie 
on the diagonal $\lbrace (t,\ldots,t)\subset \mathbb{R}^k\rbrace$ as $s(0)=0\in \mathbb{R}^k$. Particularly
$r(1)=(\tau_1(p),\ldots,\tau_k(p))$ and $\tau_1(p)=\ldots=\tau_k(p)$ and, as $p$ is arbitrary, 
$N=C_1-\ldots-C_k \subset \Delta_0$. The diagonal $\Delta_0$ and the sewed cells $N$ have the same dimensions - 
$N_0|_{U_0}=\Delta_0|_{U_0}$ on a sufficiently small negihborhood of $q_0$.

In what will follow we will study extrinsic geometry of a manifold of sewed cells as a Riemannian submanifold in the product 
$M=C_1\times\ldots\times C_k$, i.e. Gauss's, Weingarten's equations {\cite{KobNomII}}. We recall that as 
$M$ is the Riemannian product the distributions $D_i$ are totally parallel:
for a section $Y\in \Gamma(D_i)$, $\bar\nabla_XY\in \Gamma(D_i)$ - for arbitrary vector field $X$ on $M$. More
geometrically $D_i$ are invariant with respect to the parallel  displacements. In consequence, the Riemannian curvature of 
the manifold $M$ - $\bar R(X,Y) =0$ identically if $X\in\Gamma(D_i)$, $Y\in\Gamma(D_j)$, $i\neq j$.

\begin{proposition}
\label{covlifts}
 Let $\bar X\in \Gamma(D_i)$, $\bar Y\in \Gamma(D_j)$ are lifts of the vector fields $X$, $Y$ from the i-th and j-th cells res.
Then
\begin{equation}
 \bar\nabla_{\bar X}\bar Y = 
\begin{cases}
\overline{\nabla_XY}, & i=j,\\
0, & i\neq j, 
\end{cases}
\end{equation}
i.e. the covariant derivative of the lifts is a lift of the covariant derivative of the vector fields.
 \end{proposition}
\begin{proof}
The case $i\neq j$ is obvious so let assume $i=j$.
From the definition of the lift $\bar X$ it follows that $[\bar X,\bar Y] = \overline{[X,Y]}$, for the Lie
bracket of the vector fields $[X,Y]$. The Koszul's formula for the covariant derivative follows
\begin{equation*}
 \bar g(\bar\nabla_{\bar X}\bar Y,\bar Z) = g(\nabla_X Y,Z)\circ\pi_i,
\end{equation*}
so $\pi_{i*}(\bar\nabla_{\bar X}\bar Y)=\nabla_X Y$, and $\bar\nabla_{\bar X}\bar Y\in \Gamma(D_i)$, therefore 
by the uniqueness $\bar\nabla_{\bar X}\bar Y =\overline{\nabla_XY}$.
 \end{proof}
\begin{corollary}
\label{curvlifts}
\begin{equation}
 \bar R(\bar X,\bar Y)\bar Z= \overline{R(X,Y)Z},
\end{equation}
whenever the right hand makes sense.
\end{corollary}

The vector fields
\begin{equation*}
 \begin{array}{l}
 \dfrac{1}{\sqrt{k}}\bar\xi_1+\ldots+\dfrac{1}{\sqrt{k}}\bar\xi_k=\bar\xi,\\[+6pt]
          \dfrac{1}{\sqrt{k}}\bar\xi_1+\ldots+\dfrac{1}{\sqrt{k}}\bar\xi_{l-1}-\dfrac{l-1}{\sqrt{k}}\bar\xi_l,\quad l=2,\ldots,k 
 \end{array}
\end{equation*}
are pairwise orthogonal. Let $(\bar\xi, u_2\ldots,u_k)$,  be the respective orthonormal (global) frame. 
Let $N=C_1-\ldots-C_k$ are sewed cells passing through $q_0\in M$. Denote by $g$ the induced metric on $N$,
$g=\bar g|_N$, $\xi=\bar\xi|_N$, $\nabla_XY$ the Levi-Civita connection on $N$, for $X$, $Y$ tangent to $N$
\begin{eqnarray*}
 \bar\nabla_X Y &=& \nabla_X Y +h(X,Y), \\
 \bar\nabla_Xu_\alpha &=& -S_\alpha X + D^\perp_Xu_\alpha,
\end{eqnarray*}
$h(X,Y)$ is the second fundamental form of $N$, $S_\alpha$ - the Weingarten operators associated to the normal frame $(u_2,..,u_k)$,
 and $D^\perp$ is the normal connection. The  frame $(u_2,\ldots,u_k)$ is globally defined hence the normal vector bundle is parallelizable. 

\begin{proposition}
 The normal connection $D^\perp$ is flat. 
\end{proposition}

\begin{proof}
It is enough to prove $\bar g(\bar\nabla_Xu_\alpha,u_\beta)=0$, $\alpha,\beta=2,\ldots,k$. Let extend $X$ to a vector field on $M$,
for $u_\alpha=u_\alpha^1\bar\xi_1+\ldots+u_\alpha^k\bar\xi_k$, $u_\alpha^j=const$, a tensor field $\bar g(\bar\nabla_Xu_\alpha,Y)$  is symmetric
\begin{equation*}
 \bar g(\bar\nabla_Xu_\alpha,u_\beta) = \bar g(\nabla_{u_\beta}u_\alpha,X) = 
\sum\limits_{i,j=1}^ku_\beta^iu_\alpha^j\bar g(\nabla_{\bar\xi_i}\bar\xi_j,X)=0,
\end{equation*}
for $\nabla_{\bar\xi_i}\bar\xi_j=0$ for $i\neq j$ and by the Propositions \ref{covlifts}, \ref{cellfeat} 
\begin{equation*}
 \bar\nabla_{\bar\xi_j}\bar\xi_j = \overline{\nabla_{\xi_j}\xi_j}=0.
\end{equation*}
 \end{proof}

\begin{corollary} 
\label{weinxi}
 \begin{equation*}
 S_\alpha X = -\bar\nabla_Xu_\alpha, \quad \alpha=2,\ldots,k,
\end{equation*}
moreover similar arguments as in the proof of the above Proposition show that
\begin{equation*}
 S_\alpha\xi=0,\quad\alpha=2,\ldots,k.
\end{equation*}
\end{corollary}

Let $R$ be the curvature  of $N$.
\begin{proposition}
 \begin{equation}
 \bar R(X,Y)\xi = R(X,Y)\xi,
\end{equation}
that is, $R(X,Y)\xi$ is simply the restriction of $\bar R(X,Y)\xi$ to $N$. 
\end{proposition}
\begin{proof}
By the Gauss  equation {\cite{KobNomII}} the tangent part of $\bar R(X,Y)\xi$ is equal to (cf. Corollary \ref{weinxi})
\begin{equation}
 R(X,Y)\xi +\sum\limits_{\alpha=2}^k\left(h^\alpha(X,\xi)S_\alpha Y -h^\alpha(Y,\xi)S_\alpha X\right)=
R(X,Y)\xi,
\end{equation}
for $h^\alpha(X,\xi)=g(S_\alpha X,\xi)=g(X,S_\alpha \xi)=0$.
The normal part of $\bar R(X,Y)\xi$ (the Wiengarten equation)
\begin{equation}
 \begin{array}{l}
  (\widetilde\nabla_X h)(Y,\xi) - (\widetilde\nabla_Y h)(X,\xi) = \\
= \sum\limits_{\alpha=2}^k\lbrace (\nabla_X h^\alpha)(Y,\xi)-(\nabla_Y h^\alpha)(X,\xi)\rbrace u_\alpha \\
+ \sum\limits_{\alpha=2}^k\lbrace h^\alpha(Y,\xi)D_Xu_\alpha -h^\alpha(X,\xi)D_Yu_\alpha\rbrace,
 \end{array}
\end{equation}
again let extend $X$,$Y$,$\xi$, $u_\alpha$ to a vector fields on $M$
\begin{equation*}
\bar g(\bar R(X,Y)\bar\xi,u_\alpha)= \dfrac{1}{\sqrt{k}}\sum\limits_{i,j=1}^k u_\alpha^j\;\bar g(\bar R(X,Y)\bar\xi_i,\bar\xi_j)=0, 
\end{equation*}
and as the normal connection is flat and $D_Xu_\alpha=0$, $\alpha=2,\ldots,k$ the normal component of $\bar R(X,Y)\xi$ vanishes 
identically.
 \end{proof}

\section{Almost cosymplectic and almost $\alpha$-Kenmotsu sewed cells}
Let $M=C_1\times\ldots\times C_k$ be the product of cells and $(f,\bar g,\bar\xi_1,\ldots,\bar\xi_k)$ its 
canonical almost metric $f$-structure,  $N=C_1-\ldots-C_k$ are sewed cells,
by the construction the submanifold $N$ is $f$-invariant.  Again $\iota$ denotes the inclusion map.

\begin{theorem}
 The tensor fields $\varphi=f|_N$, $\xi=\bar\xi|_N$, $\eta=\sqrt{k}\,\iota^*\bar\eta_1$, $g=\bar g|_N$ 
define an almost contact metric structure on sewed cells $N$. Moreover
\begin{equation}
 d\eta=0, \quad \eta\wedge d\varPhi=0,
\end{equation}
where $\varPhi$ is the fundamental form of $N$.
\end{theorem}
\begin{proof}
At first 
\begin{equation}
g( X,\varphi Y)= \bar g(\iota_* X,\iota_*(\varphi Y)) = \bar g(\iota_*X,f\iota_*( Y)), 
\end{equation}
and $\varPhi = \iota^*\bar\varPhi$, similarly we prove $\varphi\xi=0$, 
\begin{equation}
 d\varPhi = \iota^*(d\bar\varPhi)=\iota^*\lbrace \sum\limits_{i=1}^k\pi_i^*d\varPhi_i\rbrace,
\end{equation}
 $\varPhi_i$ are the fundamental forms of cells. Each cell is a 3-dimensional manifold hence
$d\varPhi_i=2\lambda_i\eta_i\wedge\varPhi_i$ for a function $\lambda_i$ on the cell. 
\begin{equation}
 d\varPhi = 2\iota^*\left\lbrace\sum\limits_{i=1}^k\bar\lambda_i\bar\eta_i\wedge\bar\varPhi\right\rbrace= 
2\sum\limits_{i=1}^k\iota^*(\bar\eta_i)\wedge\iota^*(\bar\lambda_i\bar\varPhi_i),
\end{equation}
$\bar\lambda_i = \lambda_i\circ\pi_i$, from $\iota^*\bar\eta_1=\ldots=\iota^*\bar\eta_k=\frac{1}{\sqrt{k}}\,\eta$ (eq. \ref{alletas}) it follows
\begin{equation}
 \eta\wedge d\varPhi=\dfrac{2}{\sqrt{k}}\,\eta\wedge\eta\wedge\iota^*\left\lbrace\sum\limits_{i=1}^k\bar\lambda_i\bar\varPhi_i\right\rbrace= 0,
\end{equation}
for the further reference the form in the curly brackets we denote as $\varPhi'$.

If $\varphi X=0$ for $X$ tangent to $N$ then 
$f\iota_*(X)=0$ hence $\iota_* (X)= \alpha_j\bar\xi_j$ for real numbers $\alpha_j$, however $\bar\eta_1(\iota_*(X))=\ldots=\bar\eta_k(\iota_*(X))$
implies $\alpha_1=\ldots=\alpha_k=\alpha$ and $X=\alpha\xi$, $\varphi^3+\varphi=0=\varphi(\varphi^2+Id)$, so 
$\varphi^2 X + X=\alpha(X)\xi$ and $g(X,\xi)=\alpha(X)=\eta(X)$.
\end{proof}

The function $\lambda_i$ we will call a weight of a cell. The family $(\lambda_1,\lambda_2,\ldots,\lambda_k)$ the common weight of sewed cells. 
The next theorem is an almost a direct consequence of the above statement.
\begin{theorem}
 The sewed almost cosymplectic cells is an almost cosymplectic manifold, the sewed almost $\alpha_0$-Kenmotsu cells  is 
an almost $(\alpha_0/\sqrt{k})$-Kenmotsu manifold.
\end{theorem}
\begin{proof}
 The common weight of the family of almost cosymplectic cells is $(0,\ldots,0)\in \mathbb{R}^k$, the form $\varPhi'=0$. If we have family 
of an almost $\alpha$-Kenmotsu cells  with the same weight, 
$d\varPhi_i=2\alpha_i\eta_i\varPhi_i$, $\alpha_1=\ldots=\alpha_k=\alpha_0$, so the common weight is $(\alpha_0,\ldots,\alpha_0)\in \mathbb{R}^k$, then
$\varPhi'=\alpha_0\varPhi$. The sewed cells in this case satisfy $d\varPhi=\dfrac{2\alpha_0}{\sqrt{k}}\eta\wedge\varPhi$ - it is 
an almost $(\alpha_0/\sqrt{k})$-Kenmotsu manifold. 
\end{proof}

\begin{proposition}  
\label{cellsaux}
Assume that each cell $C_i$, $i=1,\ldots,k$, satisfies the $(\kappa,\mu,\mu')$-nullity 
condition
\begin{eqnarray*}
 R_i(X,Y)\xi_i &=& \kappa_i(\eta_i(Y)X-\eta_i(X)Y)+\mu_i(\eta_i(Y)h_iX-\eta_i(X)h_iY)+ \\ 
         && + \mu_i'(\eta_i(Y)h'_iX-\eta_i(X)h'_iY), \quad i=1,\ldots,k,
\end{eqnarray*}
$(\kappa_i,\mu_i,\mu'_i)\in\mathbb{R}^3$, $i=1,\ldots,k$.
Then the sewed cells $N=C_1-\ldots-C_k$ satisfies the following condition
\begin{eqnarray*}
 R(X,Y)\xi &=& (\eta(Y)PX-\eta(X)PY) +(\eta(Y)H_1X-\eta(X)H_1Y)+ \\[+4pt]
 && +(\eta(Y)H_2X-\eta(X)H_2Y),\\
\end{eqnarray*}
and the affinors $P$, $H_1$, $H_2$ satisfy the following commutations relations
\begin{equation*}
\begin{array}{l}
g(PX,Y)=g(X,PY), \quad g(H_iX,Y)=g(X,H_iY), \quad i=1,2,  \\[+4pt]
 P\varphi=\varphi P,\quad H_i\varphi+\varphi H_i=0, \quad PH_i=H_iP,\quad i=1,2, \\[+4pt]
H_1\xi=H_2\xi=P\xi =0.
\end{array}
\end{equation*}
\end{proposition}

\begin{proof}
 Let $(M,f,g,\bar\xi_1,\ldots,\bar\xi_k)$ be the product of the cells $C_i$
\begin{equation}
 \bar R(X,Y)\bar\xi = \dfrac{1}{\sqrt{k}}\sum\limits_{i=1}^k\bar R(X,Y)\bar\xi_i,
\end{equation}
for vector fields $X$, $Y$ on $M$.  For the lifts $\bar X,\bar Y\in \Gamma(D_i)$, cf. Corollary (\ref{curvlifts})
\begin{equation}
\begin{array}{rcl}
 \bar R(\bar X,\bar Y)\bar \xi_i &=& \kappa_i \left(\bar\eta_i(\bar Y)\bar X-\bar\eta_i(\bar X)\bar Y\right) 
+\mu_i\left( \bar\eta_i(\bar Y)\bar h_i\bar X - \bar\eta_i(\bar X)\bar h_i \bar Y\right)+ \\
 && +\mu'_i\left( \bar\eta_i(\bar Y)\bar h'_i\bar X-\bar\eta_i(\bar X)\bar h'_i\bar Y\right),
\end{array}
\end{equation}
let define the tensor fields $\bar P$, $\bar H_1$, $\bar H_2$ on $M$
\begin{equation}
\label{phi}
\begin{array}{rcl}
 \bar P &=& \dfrac{1}{k}\sum\limits_{i=1}^k -\kappa_i\bar\varphi_i^2, \\
 \bar H_1 &=& \dfrac{1}{k}\sum\limits_{i=1}^k \mu_i\bar h_i, \\
 \bar H_2 &=& \dfrac{1}{k}\sum\limits_{i=1}^k \mu'_i\bar h'_i
\end{array}
\end{equation}
Notice $\bar H_1\bar\xi_j=\bar H_2\bar\xi_j=\bar P\bar\xi_j=0 $, $j=1,\ldots,k$,  thus $Im(\bar H_1)\subset Im(f)$, 
$Im(\bar H_2)\subset Im(f)$, $Im(\bar P)\subset Im(f)$ -  the manifold of sewed cells 
 is invariant with respect to $\bar P$, $\bar H_1$,  $\bar H_2$, and these tensors give rise to the properly defined 
affinors on $N$ which we shall denote $P=\bar P|_N$, $H_1=\bar H_1|_N$, $H_2=\bar H_2|_N$.

For a vectors $X,Y$ tangent to $N$, $\bar\eta_1(X)=\ldots=\bar\eta_k(X)=\dfrac{\eta(X)}{\sqrt{k}}$ 
\begin{equation}
\begin{array}{rcl}
  \bar R(X,Y)\bar\xi &=&\dfrac{1}{\sqrt{k}}\sum\limits_{i=1}^k  \kappa_i\left(\bar\eta_i(Y)(-\bar\varphi_i^2)X-\eta(X)(-\bar\varphi_i^2)Y \right) + \\ [+4pt]
  && +\,\dfrac{1}{\sqrt{k}}\sum\limits_{i=1}^k\mu_i\left( \bar\eta_i(Y)\bar h_iX-\bar\eta_i(X)\bar h_iY \right) + \\[+4pt]
  && + \,\dfrac{1}{\sqrt{k}}\sum\limits_{i=1}^k\mu'_i( \bar\eta_i(Y)\bar h'_iX-\bar\eta_i(X)\bar h'_iY)  =\\[+10pt]
  &=& \left(\eta(X)\bar PY-\eta(Y)\bar PX\right)+\left(\eta(Y)\bar H_1X-\eta(X)\bar H_1Y\right)+ \\[+4pt]
 && +\left(\eta(X)\bar H_2X-\eta(Y)\bar H_2X\right), 
 \end{array}
\end{equation}
and the restriction
\begin{equation*}
\begin{array}{rcl}
 R(X,Y)\xi &=& \bar R(X,Y)\bar\xi|_N= (\eta(Y)PX-\eta(X)PY)+ \\ 
 && + (\eta(Y)H_1X-\eta(X)H_1Y)+  (\eta(Y)H_2X-\eta(X)H_2Y).
\end{array}
\end{equation*}
Symmetries of the tensor fields $P$, $H_1$, $H_2$ are direct consequences of the symmetries of their counterparts 
$\bar P$, $\bar H_1$, $\bar H_2$, eg. 
\begin{equation*}
 \bar H_1 f = \dfrac{1}{k}\sum\limits_{i=1}^k\mu_i\bar h_i\bar\varphi_i= -f \bar H_1.
\end{equation*}

\end{proof}

\begin{theorem}
\label{cellsth2}
 Let $N=C_1-\ldots-C_k$ are are sewed cells where each $C_i$, $i=1,\ldots,k$ is a copy of a 3-dimensional almost cosymplectic 
manifold  $M^3$, res. a 3-dimensional almost $\alpha$-Kenmotsu manifold. 
If $M^3$ satisfies the $(\kappa_0,\mu_0,\mu'_0)$-nullity condition then the  manifold $N$ of sewed cells 
is an almost cosymplectic, res.  an almost $(\alpha/\sqrt{k})$-Kenmotsu 
manifold which satisfies the $(\frac{\kappa_0}{k},\frac{\mu_0}{\sqrt{k}},\frac{\mu'_0}{\sqrt{k}})$-nullity condition. 
\end{theorem}

\begin{lemma}
\begin{equation*}
 \mathcal{L}_{\bar\xi}f|_N=\mathcal{L}_\xi\varphi, \quad (\mathcal{L}_{\bar\xi}f)f|_N=(\mathcal{L}_\xi\varphi)\varphi.
\end{equation*}
\end{lemma}

\begin{proof}(of the Lemma) Let $X$ be a local vector field tangent to $N$ defined near a point $x\in N$. Extend $X$ to 
a vector field $\tilde X$ on the product of cells, we can assume $\tilde X\in \Gamma(Im(f))$, note that $f\tilde X$ is an 
extension of $\varphi X$, so
\begin{equation*}
 [\bar\xi,f\tilde X]|_N=[\xi,\varphi X], \quad [\bar\xi,\tilde X]|_N=[\xi,X],
\end{equation*}
and
\begin{equation*}
 (\mathcal{L}_{\bar\xi}f)\tilde X|_N = [\xi,\varphi X]-\varphi [\xi,X]= (\mathcal{L}_\xi\varphi)X,
\end{equation*}
 \end{proof}

\begin{proof}(of the Theorem)
  On the ambient space of the cells product  
\begin{equation}
 (\mathcal{L}_{\bar\xi}f)X = \dfrac{1}{\sqrt{k}}\sum\limits_{i=1}^k (\mathcal{L}_{\bar\xi_i}\bar\varphi_i)=
  \dfrac{2}{\sqrt{k}}\sum\limits_{i=1}^k \bar h_i X,
\end{equation}
by the  Proposition (\ref{cellsaux})
\begin{equation}
\begin{array}{c}
 \bar P = -\dfrac{\kappa_0}{k}\displaystyle\sum\limits_{i=1}^k\bar\varphi_i=-\dfrac{\kappa_0}{k}f^2,\\
\bar H_1=\dfrac{\mu_0}{k}\displaystyle\sum\limits_{i=1}^k\bar h_i= \dfrac{\mu_0}{2\sqrt{k}}\mathcal{L}_{\bar\xi}f, \\
\quad\bar H_2=\dfrac{\mu'_0}{k}\displaystyle\sum\limits_{i=1}^k\bar h'_i = \dfrac{\mu'_0}{2\sqrt{k}}(\mathcal{L}_{\bar\xi}f)f,
\end{array}
\end{equation}
$-f^2|_N=-\varphi^2=Id-\eta\otimes\xi$ and applying the Lemma
\begin{equation}
\begin{array}{l}
H_1 = \bar H_1|_N = \dfrac{\mu_0}{2\sqrt{k}}(\mathcal{L}_{\bar\xi}f)|_N= \dfrac{\mu_0}{2\sqrt{k}}\mathcal{L}_\xi\varphi=\dfrac{\mu_0}{\sqrt{k}}h, \\
H_2 = \bar H_2|_N= \dfrac{\mu'_0}{2\sqrt{k}}(\mathcal{L}_{\bar\xi}f)f|_N= \dfrac{\mu'_0}{\sqrt{k}}h',
 \end{array}
\end{equation}
again applying the  Proposition (\ref{cellsaux})
\begin{eqnarray*}
 R(X,Y)\xi &=& \dfrac{\kappa_0}{k}(\eta(Y)X-\eta(X)Y) + \dfrac{\mu_0}{\sqrt{k}}(\eta(Y)hX-\eta(X)hY)+ \\
 && + \dfrac{\mu'_0}{\sqrt{k}}(\eta(Y)h'X-\eta(X)h'Y).
\end{eqnarray*}
\end{proof}

In what will follow we will study a bit more complicated but more interesting case of the generalized $(\kappa,\mu,\mu')$-nullity 
conditions. 

The two cells $(C_1,\varphi_1,\xi_1,\eta_1,g_1)$, $(C_2,\varphi_2,\xi_2,\eta_2,g_2)$ are locally isomorphic if 
for a points $x_1\in C_1$, $x_2\in C_2$ there are neighborhoods $x_1\in U^1\subset C_1$,
$x_2\in U^2\subset C_2$ and there is a diffeomorphism $\theta:U^1\rightarrow U^2$, $\theta(x_1)=x_2$ which preserves the structures 
\begin{equation*}
 \theta_* \circ\varphi_1 = \varphi_2\circ \theta_*,\quad \xi_2=\theta_*\xi_1,\quad \eta_1=\theta^*\eta_2,\quad g_1=\theta^*g_2.
\end{equation*}

\begin{theorem}
\label{cellsth1}
Let $N=C_1-\ldots-C_k$ is a manifold of sewed, locally isomorphic cells $C_i$, $i=1,\ldots,k$. Assume that each cell
$(C_i,\varphi_i,\xi_i,\eta_i,g_i)$ satisfies a generalized $(\kappa_i,\mu_i,\mu'_i)$-nullity condition where functions 
$\kappa_i$, $\mu_i$, $\mu'_i$ are arbitrary provided $d\kappa_i\wedge\eta_i=d\mu_i\wedge\eta_i=d\mu'_i\wedge\eta_i=0$. Then 
the manifolds of sewed cells $N$ satisfies 
a generalized $(\kappa,\mu,\mu')$-nullity condition with uniquely determined functions $\kappa$, $\mu$, $\mu'$ such that 
$d\kappa\wedge\eta=d\mu\wedge\eta=d\mu'\wedge\eta=0$.
\end{theorem} 

\begin{lemma}
 Let $q_0\in N \subset M=C_1\times\ldots\times C_k$. There is a polidisc $(U_0,\bar\tau_1,\ldots,\bar\tau_k)$ centered at $q_0$, a non-empty interval $(-\epsilon,\epsilon)$, and there are functions $u_1, u_2, u_3:(-\epsilon,\epsilon)\rightarrow \mathbb{R}$, such that 
\begin{equation}
\begin{array}{l}
\bar\tau_1(U_0)=\ldots=\bar\tau_k(U_0)=(-\epsilon,\epsilon), \quad \\
\bar\kappa_i=u_1\circ\bar\tau_i,\quad \bar\mu_i=u_2\circ\bar\tau_i, 
\quad \bar\mu'_0=u_3\circ\bar\tau_i,\quad i=1,\ldots,k,
\end{array}
\end{equation}
 where $\bar\kappa_i=\kappa_i\circ\pi_i$, $\bar\mu_i=\mu_i\circ\pi_i$, $\bar\mu'_i=\mu'_i\circ\pi_i$, $i=1,\ldots,k$,
\end{lemma}

\begin{proof}(of the Lemma)
As cells $C_i$ and $C_j$, $i\neq j$ are locally isomorphic there are small discs $x^i_0=\pi_i(q_0)\in U^i_0\subset C_i$, 
$x^j_0=\pi_j(q_0)\in U^j_0\subset C_j$ and diffemorphism $\theta: U^i_0\rightarrow U^j_0$ which preserves the structures,
we can assume that $U^i_0$, $U^j_0$ are small enough that functions $\tau_i$, $d\tau_i=\eta_i|_{U^i_0}$, 
$\tau_j$, $d\tau_j=\eta_j|_{U^j_0}$ exist and $\tau_i(x^i_0)=\tau_j(x^j_0)=0$. From $d\kappa_i\wedge\eta_i=0$ and $d\kappa_j\wedge\eta_j=0$ 
follows that there are functions $u_i,u_j: (-\epsilon,\epsilon)\rightarrow\mathbb{R}$, $(-\epsilon,\epsilon)=\tau_i(U^i_0)=\tau_j(U^j_0)$, such that 
$$
\kappa_i=u_i\circ\tau_i,\quad \kappa_j=u_j\circ\tau_j.
$$
The functions $\kappa_i$, $\kappa_j$ are scalar invariants, i.e. 
$\kappa_i = \kappa_j\circ \theta$. Therefore 
\begin{equation*}
 \kappa_i=u_i\circ\tau_i, \quad \kappa_i=\kappa_j\circ \theta = u_j\circ \tau_j \circ\theta = u_j\circ \tau_i,
\end{equation*}
hence $u_i=u_j$. Similar arguments prove the existence of functions $v$, $w$, $\mu_i=v\circ\tau_i$, $\mu_j=v\circ\tau_j$,
$\mu'_i=w\circ\tau_i$, $\mu'_j=w\circ\tau_j$.

Now we fix $i=1$ and provide above construction for each pair $(C_1,C_j)$, $j=2,\ldots,k$. Thus we have family 
$(\tau_1,U^1_0)$,..,$(\tau_k,U^k_0)$, $\tau_i(U^i_0)=(-\epsilon,\epsilon)$, $i=1,\ldots,k$, and 
functions $u_1, u_2, u_3:(-\epsilon,\epsilon)\rightarrow\mathbb{R}$, $\kappa_i=u_1\circ\tau_i$, $\mu_i=u_2\circ\tau_i$, 
$\mu'_0=u_3\circ\tau_i$, $i=1,\ldots,k$
 and $(U_0=\bigcap\limits_{i=1}^k\pi_i^{-1}(U^i_0)$, $\bar\tau_i=\tau_i\circ\pi_i)$ is the required polidisc.
\end{proof}

\begin{proof}(of the Theorem)
Note that the thesis of the Proposition (\ref{cellsaux}) remains unchanged if we use the restriction of the functions $\bar\kappa_i$, 
$\bar\mu_i$, $\bar\mu'_0$, $i=1,\ldots,k$ to the submanifold of the sewed cells $N$. Locally the sewed cells 
are described by $\bar\tau_1(q)=\ldots=\bar\tau_k(q)$, by the Lemma 
\begin{equation}
\begin{array}{c}
 \bar\kappa_1|_N=\ldots=\bar\kappa_k|_N=\kappa,\\
 \bar\mu_1|_N=\ldots=\bar\mu_k|_N=\mu,\\
\bar\mu'_1|_N=\ldots=\bar\mu'_k|_N = \mu',
\end{array}
\end{equation}
and (cf. Theorem \ref{cellsth1}) $N$ satisfies the $(\dfrac{\kappa}{k},\dfrac{\mu}{\sqrt{k}},\dfrac{\mu'}{\sqrt{k}})$-nullity condition,
$d\bar\kappa_i= u_1' d\bar\tau_i = u_1'\bar\eta_i$, finally $d\bar\kappa|_N=d\kappa=u_1'\frac{\eta}{\sqrt{k}}$ and 
$d\kappa\wedge\eta=0$, similar for $\mu$, $\mu'$.
\end{proof}

\section{Examples and final remarks}
For the purposes of the next example  we recall the following  result {\cite{DacXi}}
\begin{theorem} 
On an almost cosymplectic manifold $(M,\varphi,\xi,\eta,g)$, $\text{\rm dim}\,M=2n+1$, $n\geqslant 2$, the vector field $\xi$ belongs to 
the $k$-nullity distribution, $k<0$, if and only if for a point $p\in M$ there exists 
a coordinate neighborhood $(U,(t,x^1,\ldots,x^{2n}))$, $p\in U$, on which
\begin{equation*}
 \begin{array}{rcl}
 \xi &=&\dfrac{\partial}{\partial t},\quad \eta=dt, \\[+4pt]
g &=& dt\otimes dt +e^{2\lambda t}\sum\limits_{\mu=1}^n dx^\mu\otimes dx^\mu +e^{-2\lambda t}\sum\limits_{\mu=1}^n dx^{n+\mu}\otimes dx^{n+\mu}, \\[+4pt]
\varphi &=& e^{2\lambda t}\sum\limits_{\mu=1}^n dx^\mu\otimes \dfrac{\partial}{\partial x^{n+\mu}} - 
          e^{-2\lambda t}\sum\limits_{\mu=1}^n dx^{n+\mu}\otimes \dfrac{\partial}{\partial x^\mu},
\end{array}
\end{equation*}
where $\lambda=\sqrt{|k|}$.
\end{theorem}

\begin{example} 
{\rm Let $M=C_1\times C_2$ be the product of copies of  almost cosymplectic 
cells $C_i=(\mathbb{R}^3,\varphi_i,\xi_i,\eta_i,g_i)$, $i=1,2$, 
both satisfying $\kappa$-nullity conditions with $\kappa_1=\kappa_2=-\lambda^2=\kappa_0$, 
 sewed cells is simply a hyperplane $H=\lbrace (t_1,x_1,y_1,t_2,x_2,y_2)\in M| t_1=t_2\rbrace$, and the almost contact metric structure
$(\varphi,\xi,\eta,g)$ on $H$}
\begin{equation*}
\begin{array}{rcl}
 \xi &=& \frac{1}{\sqrt{2}}\,\dfrac{\partial}{\partial s},\quad \eta=\sqrt{2}\,ds,\quad  s=t_1=t_2|_H \\[+8pt]
 g &=& 2ds\otimes ds + e^{2\lambda s}(dx_1\otimes dx_1+dx_2\otimes dx_2) + \\[+6pt]
 && + e^{-2\lambda s}(dy_1\otimes dy_1+ dy_2\otimes dy_2),\\[+6pt]
\varphi &=& e^{2\lambda s}(dx_1\otimes\dfrac{\partial}{\partial x_1}+dx_2\otimes\dfrac{\partial}{\partial x_2}) - \\
&& - e^{-2\lambda s}(dy_1\otimes\dfrac{\partial}{\partial y_1}+dy_2\otimes\dfrac{\partial}{\partial y_2})
\end{array} 
\end{equation*}
{\rm if we reparametrize $t\mapsto \frac{s}{\sqrt{2}}$  -  the obtained 
structure is almost cosymplectic and satisfies $\kappa'$-nullity condition, by the above result, and $\kappa' = -\lambda^2/2$, $\kappa'=\kappa_0/2$
.}
\end{example}

Customary in the geometry of almost $\alpha$-Kenmotsu manifolds the operator $h'$ is defined as 
\begin{equation*}
 h' = \dfrac{1}{2\alpha}(\mathcal{L}_\xi\varphi)\varphi.
\end{equation*}
Using this convention the Theorem (\ref{cellsth2}) has the form: if $C_1$,\ldots, $C_k$ are copies of an almost $\alpha$-Kenmotsu manifold $M^3$ 
and $M^3$ satisfies $(\kappa_0,\mu_0,\mu'_0)$-nullity condition then 
the manifold of sewed cells is an almost $\alpha/\sqrt{k}$-Kenmotsu manifolds which satisfies 
the $(\frac{\kappa_0}{k},\frac{\mu_0}{\sqrt{k}},\frac{\mu,_0}{k})$-nullity condition. 

\begin{example} {\cite{DilPastNulDist,Dileo1}} If $(M^{2n+1},\varphi,\xi,\eta,g)$ is an almost $\alpha$-Kenmotsu manifold such that 
 $\xi$ belongs to the $(\kappa,\mu)'$-nullity distribution ($(\kappa,\mu)'=(\kappa,0,\mu')$), and $\kappa < -\alpha^2$, then the manifold 
is locally isometric to the warped product
\begin{equation*}
 \mathbb{R}\times_f \mathbb{R}^n\times_{f'}\mathbb{R}^n,
\end{equation*}
 with warping functions 
\begin{equation*}
 f= ce^{\alpha(1+\lambda)t},\quad f'=c'e^{\alpha(1-\lambda)t}, \quad\lambda=\sqrt{-1-\frac{\kappa}{\alpha^2}}, \quad c,c'=const > 0.
\end{equation*}
here $\eta=dt$. Now in the particular case $n=1$, let take two copies $C_1$, $C_2$ of $M^3$, for simplicity
$C_1=\mathbb{R}^3\ni p=(t_1,x_1,x_2)$, $C_2=\mathbb{R}^3\ni p=(t_2,y_1,y_2)$. Then  the sewed cells $C_1-C_2$ is a hyperplane again 
$H:t_1=t_2 \subset \mathbb{R}^6=\mathbb{R}^3\times\mathbb{R}^3$, the metric
\begin{equation*}
 g=2ds\otimes ds + ce^{\alpha(1+\lambda)s}g_1+ c'e^{\alpha(1-\lambda)}g_2, \quad s:H\rightarrow \mathbb{R},\quad s=t_1|_H=t_2|_H,
\end{equation*}
 $g_i=dx_i\otimes dy_i$, $i=1,2$, again after rescaling $t=\sqrt{2}s$, $g$ is a warped product
\begin{equation*}
\mathbb{R}\times_f\mathbb{R}^2\times_f'\mathbb{R}^2,\quad
 f=ce^{\frac{\alpha}{\sqrt{2}}(1+\lambda)t}, \quad f'=c'e^{\frac{\alpha}{\sqrt{2}}(1-\lambda)t}
\end{equation*}
in consequence $(H,\varphi,\xi,\eta,g)$ satisfies $(\kappa,\mu)'$-nullity condition with 
\begin{equation*}
 \kappa = -(\frac{\alpha}{\sqrt{2}})^2(1+\lambda^2)= \dfrac{\kappa_0}{2}, \quad \mu = -2(\frac{\alpha}{\sqrt{2}})^2=\dfrac{\mu_0}{2}. 
\end{equation*}
\end{example}

\begin{example} Consider an almost contact metric manifold {\cite{Salt,Murath}} $(M^3,\varphi,\xi,\eta,g)$
\begin{equation*}
 M^3=\lbrace (x,y,z)\in \mathbb{R}^3, z >0\rbrace
\end{equation*}
an orthonormal repair $(X,\varphi X,\xi)$
\begin{equation}
 X =\dfrac{\partial}{\partial x},\quad \varphi X=\dfrac{\partial}{\partial y}, 
\quad \xi= (x -ye^{-2z})\dfrac{\partial}{\partial x}+(y-xe^{-2z})\dfrac{\partial}{\partial y}+\dfrac{\partial}{\partial z}
\end{equation}
$\eta=dz$, $M^3$ is an almost Kenmotsu manifolds which satisfies the $(\kappa,0,0)$-nullity condition with nonconstant $\kappa= -(1+e^{-4z})$.
Similar as in above Examples 
\begin{equation*}
 C_1-C_2=  \lbrace (x_1,x_2, z_1, y_1,y_2, z_2)\in \mathbb{R}^6, z_1=z_2 > 0\rbrace ,
\end{equation*}
an orthonormal frame $(X_1,\varphi X_1, X_2,\varphi X_2,\xi)$
\begin{equation}
 \begin{array}{l}
 X_1=\dfrac{\partial}{\partial x_1},\quad \varphi X_1= \dfrac{\partial}{\partial y_1}\quad 
X_2=\dfrac{\partial}{\partial x_2}, \varphi X_2=\dfrac{\partial }{\partial y_2}, \\[+10pt]
\xi = \frac{1}{\sqrt{2}}(x_1-y_1e^{-2s})\dfrac{\partial}{\partial x_1}+\frac{1}{\sqrt{2}}(x_2-y_2e^{-2s})\dfrac{\partial}{\partial x_2} + \\[+10pt]
+ \frac{1}{\sqrt{2}}(y_1-x_1e^{-2s})\dfrac{\partial}{\partial y_1} +\frac{1}{\sqrt{2}}(y_2-x_2e^{-2s})\dfrac{\partial}{\partial y_2} + 
\frac{1}{\sqrt{2}}\dfrac{\partial}{\partial s}
\end{array}
\end{equation}
where $\eta=\sqrt{2}ds$, $s=z_1|_{C_1-C_2}=z_2|_{C_1-C_2}$, the non-zero Lie brackets
\begin{equation}
 \begin{array}{c}
 [X_1,\xi]=\dfrac{1}{\sqrt{2}}X_1-\dfrac{e^{-2s}}{\sqrt{2}}\varphi X_1, \quad [\varphi X_1, \xi]=-\dfrac{e^{-2s}}{\sqrt{2}}X_1+\dfrac{1}{\sqrt{2}}\varphi X_1,\\[+10pt]
 \left[X_2,\xi\right] = \dfrac{1}{\sqrt{2}}X_2-\dfrac{e^{-2s}}{\sqrt{2}}\varphi X_2, \quad [\varphi X_2,\xi]=-\dfrac{e^{-2s}}{\sqrt{2}}X_2+\dfrac{1}{\sqrt{2}}\varphi X_2.
\end{array}
\end{equation}
If we change the frame 
\begin{equation}
\begin{array}{l}
 \bar X_1= \dfrac{X_1+\varphi X_1}{\sqrt{2}},\quad \bar X_2 = \dfrac{X_2+\varphi X_2}{\sqrt{2}}, \\[+8pt]
 \bar Y_1 = \dfrac{-X_1+\varphi X_1}{\sqrt{2}},\quad \bar Y_2 = \dfrac{-X_2+\varphi X_2}{\sqrt{2}},
\end{array}
\end{equation}
$\varphi \bar X_1 =\bar Y_1$, $\varphi \bar X_2=\bar Y_2$,  then the Lie brackets take a simpler form $i=1,2$
\begin{equation}
 \begin{array}{c}
[\bar X_i,\xi] = \dfrac{1-e^{-2s}}{\sqrt{2}}\bar X_i, \quad [\bar Y_i,\xi]= \dfrac{1+e^{-2s}}{\sqrt{2}}\bar Y_i,
\end{array}
\end{equation}
other brackets are zero, subsequently applying the Koszul's formula for the covariant derivative 
\begin{equation}
 \begin{array}{l}
 \nabla_{\bar X_i}\xi = \dfrac{1-e^{-2s}}{\sqrt{2}}\bar X_i, \quad \nabla_{\bar Y_i}\xi = \dfrac{1+e^{-2s}}{\sqrt{2}}\bar Y_i, \\[+10pt]
 \nabla_{\bar X_i}\bar X_i = -\dfrac{1-e^{-2s}}{\sqrt{2}}\xi, \quad \nabla_{\bar Y_i}\bar Y_i= -\dfrac{1+e^{-2s}}{\sqrt{2}}\xi, \\[+10pt]
 \nabla_{\xi}\bar X_i=0, \quad \nabla_{\xi}\bar Y_i=0,
\end{array}
\end{equation}
 for the curvature we obtain
\begin{equation}
 \begin{array}{rcl}
 R(\bar X_i,\xi) &=& -\nabla_\xi\nabla_{\bar X_i}\xi -\nabla_{[\bar X_i,\xi]}\xi = \\[+10pt]
            &=& -(\xi(\dfrac{1-e^{-2s}}{\sqrt{2}})+\frac{1}{2}(1-e^{-2s})^2)\bar X_i = -\dfrac{1+e^{-4s}}{2}\bar X_i, \\[+10pt]
R(\bar Y_i,\xi) &=&  -(\xi(\dfrac{1+e^{-2s}}{\sqrt{2}})+\frac{1}{2}(1+e^{-2s})^2)\bar Y_i = -\dfrac{1+e^{-4s}}{2}\bar Y_i,\\[+10pt]
&& R(\bar X_i,\bar X_j)\xi  =R(\bar X_i, \bar Y_j)\xi = R(Y_i,Y_j)\xi =0,
\end{array} 
\end{equation}
-  $C_1-C_2$ is an almost $1/\sqrt{2}$-Kenmotsu manifold which satisfies $(\kappa',0,0)$-nullity condition with the function $\kappa' = -\frac{1+e^{-4s}}{2}$,
again $\kappa'=\kappa_0/2$. 
\end{example}

Here some final remarks. The provided construction suggests to study particular submanifolds  
of almost metric $f$-manifolds with closed coframings.  
The  construction can be developed in many ways: although we take the median $\bar\xi$ as a starting point, more generally $\bar\xi$ can 
be defined by a point $c$ of the sphere  $\mathbb{S}^k \ni c=(c_1,\ldots,c_k)$, $\bar\xi= c_1\bar\xi_1+\ldots+c_k\bar\xi_k$. Cells are 3-dimensional 
manifolds, as almost contact metric  manifolds they are CR-integrable with canonically defined almost CR structure. It should be clear that sewed cells 
are also CR-integrable - which is an advantage - as we may study invariant connections  - and in the same time disadvantage - one 
cannot go beyond CR-integrable structures with this construction. One of the possible remedy is to try to sew the cells of higher dimensions, 
eg. non-CR-integrable 5-dimensional cells with 3-dimensional. Not all possible features are inherited by the sewed cells, eg. there are 3-dimensional 
examples of conformally flat almost cosymplectic manifolds {\cite{DacOlszConf}}, however  sewed such cells  are never conformally flat - 
at least they are flat and cosymplectic -  as they have K\"ahler leaves {\cite{DacOlszConf}).
Similar situation appears when one consider manifolds with a pointwise constant $\varphi$-sectional curvature: evidently each 3-dimensional 
manifold has this feature - again sewed almost cosymplectic cells are 
never manifolds of pointwise $\varphi$-sectional curvature {\cite{DacPhi}} - provided they not cosymplectic.

\end{document}